\pdfoutput=1
\documentclass{article}
\usepackage[a4paper, margin=1.5in]{geometry}  

\usepackage{times}            
\usepackage[T1]{fontenc}      
\usepackage[utf8]{inputenc}   
\usepackage{microtype}        

\usepackage{amsmath}
\usepackage{amsthm}
\usepackage{amssymb}
\usepackage{commath}
\usepackage{mathtools}
\usepackage{thmtools}
\usepackage{mathrsfs}

\usepackage{hyperref}
\hypersetup{
  colorlinks=true,
  citecolor=blue,
  linkcolor=blue,
  }
\usepackage{cleveref}
\usepackage[numbers]{natbib}
\usepackage[dvipsnames]{xcolor} 

\numberwithin{equation}{section}
\declaretheorem[Refname={Theorem,Theorems}]{theorem}
\numberwithin{theorem}{section} 

\declaretheorem[style=theorem,numberlike=theorem,Refname={Example,Examples}]{example}
\declaretheorem[style=definition,numberlike=theorem,Refname={Remark,Remarks}]{remark}

\declaretheorem[name=Proposition,numberlike=theorem,Refname={Proposition,Propositions}]{proposition}

\DeclarePairedDelimiterX\Set[2]{\lbrace}{\rbrace}%
{ #1 \,:\, #2 }                                         

\newcommand{\R}{\mathbb{R}} 
\newcommand{\N}{\mathbb{N}} 
\newcommand{\Z}{\mathbb{Z}} 
\newcommand{\C}{\mathbb{C}} 

\newcommand{\rev}[1]{{\color{black}{#1}}}

\title{\Large \textbf{Orthonormal Expansions for Translation-Invariant Kernels}}
\author{Filip Tronarp$^1$ --- Toni Karvonen$^2$}
\date{
  {
    \normalsize
    $^1$Centre for Mathematical Sciences, Lund University, Sweden \\
    $^2$Department of Mathematics and Statistics, University of Helsinki, Finland \\
  }
  \vspace{0.4cm}
  \today
}

\begin{document}

\maketitle

\begin{abstract}
  \noindent
  We present a general Fourier analytic technique for constructing orthonormal basis expansions of translation-invariant kernels from orthonormal bases of $\mathscr{L}_2(\mathbb{R})$.
  This allows us to derive explicit expansions on the real line for (i) Mat{\'e}rn kernels of all half-integer orders in terms of associated Laguerre functions, (ii) the Cauchy kernel in terms of rational functions, and (iii) the Gaussian kernel in terms of Hermite functions.
\end{abstract}

\section{Introduction}
\label{sec:introduction}

Let $\Omega$ be a vector space.
A symmetric positive-semidefinite kernel $r \colon \Omega \times \Omega \to \R$ is \emph{translation-invariant} if $r(t, u) = \Phi(t - u)$ for some $\Phi \colon \Omega \to \R$ and all $t, u \in \Omega$.
Translation-invariant kernels, also known as \emph{stationary} or \emph{shift-invariant} kernels, are a mainstay of radial basis function interpolation~\cite{Wendland2005} and Gaussian process modelling as used in, for example, spatial statistics~\cite{Stein1999} and machine learning~\cite{RasmussenWilliams2006}.
Each positive-semidefinite kernel induces a unique reproducing kernel Hilbert space (RKHS), $\mathscr{H}_r(\Omega)$, which is equipped with an inner product $\langle \cdot, \cdot \rangle_r$ and the associated norm $\norm[0]{\cdot}_r$~\rev{\citep[e.g.,][Section~2.2]{Paulsen2016}}.
Practically every commonly used kernel induces an infinite-dimensional RKHS that is separable (see~\cite{OwhadiScovel2017} for a short review on separability of RKHSs), which means that $\mathscr{H}_r(\Omega)$ has an orthonormal basis $\{\psi_m\}_{m \in I}$ for some countably infinite index set $I$ (e.g., $I = \N$) and that the kernel admits the pointwise convergent orthonormal expansion
\begin{equation}
  \label{eq:kernel-expansion-intro}
  r(t, u) = \sum_{m \in I} \psi_m^*(t) \psi_m(u) \quad \text{ for all } \quad t, u \in \Omega,
\end{equation}
\rev{where $z^*$ denotes the complex conjugate of $z \in \C$.}
If $\Omega$ is a compact subset \rev{of the Euclidean space} $\R^d$ and $r$ is continuous, the expansion~\eqref{eq:kernel-expansion-intro} converges uniformly~\citep[Section~11.3]{Paulsen2016}.
Expansions of the form~\eqref{eq:kernel-expansion-intro} are often needed to develop reduced rank methods of sub-cubic computational complexity~\cite{RahimiRecht2007,SolinSarkka2020}, to improve numerical stability~\cite{FasshauerMcCourt2012}, and for various theoretical purposes~\cite[e.g.,][]{Karvonen2022, Steinwart2019}.

However, few orthonormal expansions appear to have been constructed for translation-invariant kernels.
To the best of our knowledge, the Matérn-$\frac{1}{2}$ kernel $r(t, u) = \exp(-\lambda \abs[0]{t - u})$ and the Gaussian kernel $r(t, u) = \exp(-\frac{1}{2}\lambda^2(t-u)^2)$ on subsets of $\R$ are the only commonly used translation-invariant kernels for which orthonormal expansions have been found.
For various expansions of the Matérn-$\frac{1}{2}$ kernel, see Section~4 in \cite{Hawkins1989}, Section~3.4.1 in \cite{VanTrees2001}, Example~4.1 in \cite{Xiu2010}, and Example~2.5 and Appendix~A.2 in \cite{FasshauerMcCourt2015}.
For the Gaussian kernel both a simple non-Mercer expansion based on a Taylor expansion of the exponential function~\cite[e.g.,][]{Minh2010} and a class of Mercer expansions~\cite[Section~12.2.1]{FasshauerMcCourt2015}, which appear to have originated in~\cite[Section~4]{Zhu1998}, are available.
A large collection of expansions for kernels which are not translation-invariant can be found in~\cite[Appendix~A]{FasshauerMcCourt2015}.
In this article we describe a general and conceptually simple Fourier analytic technique, contained in Theorem~\ref{thm:main-theorem}, for constructing orthonormal bases for translation-invariant kernels on $\R$ out of orthonormal bases of $\mathscr{L}_2(\R)$.
We then use this technique to compute orthonormal expansions for three commonly used classes of kernels.

Ours is what one could call a \emph{kernel-centric} approach.
That is, our starting point is a kernel that has, in some sense, desirable or intuitive properties and our goal is to \rev{find} its orthonormal expansion.
The \emph{space-centric} approach is to start with a Hilbert space or its orthonormal basis, show that this space is an RKHS, and construct its reproducing kernel via~\eqref{eq:kernel-expansion-intro}; under fortuitous circumstances the kernel is available in closed form.
A prime example of this approach is how Korobov spaces and their kernels, which can be expressed in terms of Bernoulli polynomials, are used in the quasi-Monte Carlo literature~\cite[e.g.,][Section~5.8]{DickKuoSloan2013}.
Other examples include Hardy spaces~\cite[Section~1.4.2]{Paulsen2016}, power series kernels~\cite{Zwickngaslc2009}, and Hermite spaces~\cite{IrrgeherLeobacher2015}.
Our technique to construct orthonormal bases is similar to the method in~\cite{NovakUllrich2018}, where the goal is however to find a closed form expression for the reproducing kernel of a Hilbert space.

\subsection{Construction of orthonormal bases}

\rev{Let $\abs[0]{z}$ denote the modulus of $z \in \C$ and recall that $z^*$ is the complex conjugate.}
The spaces $\mathscr{L}_2(\R)$ and $\mathscr{L}_2(\R, 1/2\pi)$ consist of all square-integrable functions $f \colon \R \to \C$ and are equipped with the inner products
\begin{equation*}
  \langle f, g \rangle_{\mathscr{L}_2(\R)} = \int_{-\infty}^\infty f^*\!(t) g(t) \dif t \quad \text{ and } \quad \langle f, g \rangle_{\mathscr{L}_2(\R, 1/2\pi)} = \frac{1}{2\pi} \int_{-\infty}^\infty f^*\!(t) g(t) \dif t.
\end{equation*}
The Fourier transform and the corresponding inverse transform for any integrable or square-integrable function $f$ are defined as
\begin{equation*}
\hat{f}(\omega) = \int_{-\infty}^\infty f(t) e^{-i\omega t} \dif t \quad \text{ and } \quad f(t) = \frac{1}{2\pi} \int_{-\infty}^\infty \hat{f}(\omega) e^{i\omega t}\dif \omega.
\end{equation*}
The Fourier transform defines an isometry from $\mathscr{L}_2(\mathbb{R})$ to $\mathscr{L}_2(\mathbb{R},1/2\pi)$ via the Plancherel theorem
\begin{equation*}
\int_{-\infty}^\infty f^*\!(t) g(t) \dif t = \frac{1}{2\pi} \int_{-\infty}^\infty \hat{f}^*\!(\omega) \hat{g}(\omega)  \dif \omega.
\end{equation*}
The functions $f$ and $\hat{f}$ are referred to as time domain and Fourier domain representations, respectively.
Our $\mathscr{H}_r(\R)$-orthonormal expansions are derived from the following rather straight-forward theorem.
Let $I$ be a countably infinite index set, typically either $\N$ or $\Z$.

\begin{theorem}[Construction of orthonormal bases]
  \label{thm:main-theorem}
  \rev{Suppose that $r(t, u) = \Phi(t - u)$ is a translation-invariant symmetric positive-definite kernel with $\Phi \in C(\R) \cap \mathscr{L}_1(\R)$.
  Let $\{\varphi_m\}_{m \in I}$ be an orthonormal basis of $\mathscr{L}_2(\R)$ and $h$ a function such that $\abs[0]{\hat{h}(\omega)} = \hat{\Phi}(\omega)^{1/2}$.}
  Then the functions
  \begin{equation*}
    \psi_m(t) = \int_{-\infty}^\infty h(t - \tau) \varphi_m(\tau) \dif \tau \quad \text{ with Fourier transforms } \quad \hat{\psi}_m(\omega) = \hat{h}(\omega) \hat{\varphi}_m(\omega)
  \end{equation*}
  for $m \in I$ form an orthonormal basis of $\mathscr{H}_r(\R)$ and the kernel $r$ has the pointwise convergent expansion
  \begin{equation} \label{eq:main-theorem-r-expansion}
    r(t, u) = \sum_{m \in I} \psi_m^*(t) \psi_m(u) \quad \text{ for all } \quad t, u \in \R.
  \end{equation}
\end{theorem}
\begin{proof}
\rev{That $r$ is symmetric positive-definite implies that $\hat{\Phi}$ is real-valued and positive~\cite[Theorem~6.11]{Wendland2005}.}
For a function $h$ such that $\abs[0]{\hat{h}(\omega)} = \hat{\Phi}(\omega)^{1/2} > 0$ for all $\omega \in \R$ we define a convolution operator $\mathcal{H} \colon \mathscr{L}_2(\R) \to \mathscr{L}_2(\R)$ via
\begin{equation*}
  (\mathcal{H} f)(t) = \int_{-\infty}^\infty h(t - \tau) f(\tau) \dif \tau \quad \text{ for all } \quad t \in \R.
\end{equation*}
Note that the convolution theorem yields $\widehat{\mathcal{H} f}(\omega) = \hat{h}(\omega) \hat{f}(\omega)$.
By the standard characterisation (see~\cite{KimeldorfWahba1970} or~\cite[Theorem~10.12]{Wendland2005}) of the RKHS of a translation-invariant kernel,
\begin{equation}
  \label{eq:RKHS-fourier}
  \langle f, g \rangle_r = \frac{1}{2\pi} \int_{-\infty}^\infty \frac{\hat{f}^*\!(\omega) \hat{g}(\omega)}{\hat{\Phi}(\omega)} \dif \omega \quad \text{ for any } \quad f, g \in \mathscr{H}_r(\R).
\end{equation}
For any $f, g \in \mathscr{L}_2(\R)$ the convolution theorem and Plancherel theorem thus give
\begin{equation*}
    \langle \mathcal{H}f, \mathcal{H}g \rangle_r = \frac{1}{2\pi} \int_{-\infty}^\infty \frac{\abs[0]{\hat{h}(\omega)}^2 \hat{f}^*\!(\omega) \hat{g}(\omega)}{\hat{\Phi}(\omega)} \dif \omega = \frac{1}{2\pi} \int_{-\infty}^\infty \hat{f}^*\!(\omega) \hat{g}(\omega) \dif \omega = \langle f, g \rangle_{\mathscr{L}_2(\R)},
\end{equation*}
which shows that $\mathcal{H}$ is an isometry from $\mathscr{L}_2(\R)$ to $\mathscr{H}_r(\R)$.
It follows from~\eqref{eq:RKHS-fourier} that the inverse Fourier transform
\begin{equation*}
  (\mathcal{H}^{-1} f)(t) = \frac{1}{2\pi} \int_{-\infty}^\infty \frac{\hat{f}(\omega)}{\hat{h}(\omega)} e^{i\omega t} \dif \omega \quad \text{ for all } \quad t \in \R
\end{equation*}
defines the inverse of $\mathcal{H}$.
Therefore $\mathcal{H}$ is an isometric isomorphism and thus maps every orthonormal basis of $\mathscr{L}_2(\R)$ to an orthonormal basis of $\mathscr{H}_r(\R)$~\cite[Section~2.6]{Higgins1977}.
\rev{The kernel has a pointwise convergent expansion of the form~\eqref{eq:main-theorem-r-expansion} for every orthonormal basis of $\mathscr{H}_r(\R)$~\citep[Theorem~2.4]{Paulsen2016}.}
\end{proof}

To obtain the basis functions $\psi_m$ in time domain using \Cref{thm:main-theorem} one has to either compute the convolution $\int_{-\infty}^\infty h(t - \tau) \rev{\varphi_m(\tau)} \dif \tau$ or the inverse Fourier transform of $\hat{h}(\omega) \hat{\varphi}_m(\omega)$.
It is therefore necessary to select a basis of $\mathscr{L}_2(\R)$ for which either of these operations can be done in closed form.
We use \Cref{thm:main-theorem} to derive orthonormal expansions for (i) Matérn kernels for all half-integer orders, (ii) the Cauchy kernel (i.e., rational quadratic kernel~\cite[Equation~(4.19)]{RasmussenWilliams2006} with $\alpha = 1$), and (iii) the Gaussian kernel.
The expansions are summarised in \Cref{sec:summary}.
\rev{All expansions appearing in this article converge pointwise.}

\subsection{On Mercer expansions}

Let $\Omega$ be a subset of $\R^d$ and $w \colon \Omega \to [0, \infty)$ a weight function.
The Hilbert space $\mathscr{L}_2(\Omega, w)$ is equipped with the inner product
\begin{equation*}
  \langle f, g \rangle_{\mathscr{L}_2(\Omega, w)} = \int_\Omega f^*\!(t) g(t) w(t) \dif t
\end{equation*}
and consists of all functions $f \colon \R \to \C$ for which the corresponding norm is finite.
Suppose that the kernel $r$ is continuous and define the integral operator
\begin{equation}
  \label{eq:mercer-integral-operator}
  \mathcal{T}_{r,w} f = \int_{-\infty}^\infty r(\cdot, u) f(u) w(u) \dif u.
\end{equation}
Under certain assumptions, Mercer's theorem~\cite{SteinwartScovel2012} states that (i) $\mathcal{T}_{r,w}$ has continuous eigenfunctions $\{\rev{\vartheta}_m\}_{m=0}^\infty$ and corresponding positive non-increasing eigenvalues $\{\mu_m\}_{m=0}^\infty$ which tend to zero, (ii) $\{\rev{\vartheta}_m\}_{m=0}^\infty$ are an orthonormal basis of $\mathscr{L}_2(\Omega, w)$, and (iii) $\{\sqrt{\smash[b]{\mu_m}} \rev{\vartheta}_m\}_{m=0}^\infty$ is an orthonormal basis of $\mathscr{H}_r(\Omega)$.
Consequently, the kernel has the pointwise convergent \emph{Mercer expansion}
\begin{equation}
  \label{eq:kernel-expansion-mercer}
  r(t, u) = \sum_{m = 0}^\infty \mu_m \rev{\vartheta}_m^*(t) \rev{\vartheta}_m(u) \quad \text{ for all } \quad t, u \in \Omega.
\end{equation}

While Mercer's theorem and the eigenvalues of $\mathcal{T}_{r,w}$ constitute a powerful tool for understanding topics such as optimal approximation in $\mathscr{L}_2(\Omega, w)$-norm (e.g.,~\cite[Corollary~4.12]{NovakWozniakowski2008} and~\cite[Section~2.4]{FasshauerHickernell2012}) and improved approximation orders in subsets of $\mathscr{H}_r(\Omega)$~\cite[Section~11.5]{Wendland2005}, both in theoretical research and practical applications there is often no reason to prefer a Mercer expansion~\eqref{eq:kernel-expansion-mercer} over a generic RKHS-orthonormal expansion~\eqref{eq:kernel-expansion-intro}.
For example, the Karhunen--Lo\`eve theorem is merely a special case of a more general result that a Gaussian process with covariance kernel~$r$ can be expanded in terms of any orthonormal basis of $\mathscr{H}_r(\Omega)$~\cite[Chapter~III]{Adler1990}.
When an expansion is being sought \rev{solely for} computational reasons, it does not matter whether or not this expansion is Mercer.

Constructing a Mercer expansion by first identifying a convenient weight and then finding the eigendecomposition of the integral operator~\eqref{eq:mercer-integral-operator} can be rather involved, which is illustrated by the construction in~\cite[Example~2.5]{FasshauerMcCourt2015} for the Matérn-$\frac{1}{2}$ kernel.
What makes \Cref{thm:main-theorem} convenient is therefore that it does \emph{not} require that the expansion be Mercer for some weight.
However, identifying a weight $w$ for which the basis function $\psi_m$ constructed via \Cref{thm:main-theorem} are $\mathscr{L}_2(\R, w)$-orthogonal shows that the expansion is Mercer because the $\mathscr{L}_2(\R, w)$-normalised versions of $\psi_m$ are the eigenfunctions of $\mathcal{T}_{r,w}$.
It turns out that our expansion for the Gaussian kernel is Mercer and the ones for Matérn kernels are ``almost'' Mercer, in that all but finitely many basis functions are orthogonal in $\mathscr{L}_2(\R, w)$ for a certain weight.

\section{Summary of expansions}
\label{sec:summary}

This section summarises the expansions that we derive using \Cref{thm:main-theorem}.
Each expansion converges pointwise for all $t, u \in \R$.
All expansions are for kernels with unit scaling.
Expansions of arbitrary scalings, $\lambda$, may be obtained by considering the kernel $r(\lambda t, \lambda u)$, for which the corresponding basis functions are $\psi_m(\lambda t)$.

\subsection{Matérn kernels}
Expansions for Matérn kernels are derived in \Cref{sec:matern}.
A Matérn kernel of order $\alpha > 0$ is
\begin{equation} \label{eq:matern-kernel-intro}
r_{\alpha}(t, u) =   \frac{2^{1-\alpha}}{\Gamma(\alpha)} ( \abs[0]{t-u})^\alpha \mathrm{K}_\alpha( \abs[0]{t-u}  ),
\end{equation}
where $\Gamma$ is the Gamma function and $\mathrm{K}_\alpha$ the modified Bessel function of the second kind of order~$\alpha$.
Let $\mathrm{L}_m^{(\eta)}$ denote the $m$th associated Laguerre polynomial of index $\eta$, defined in~\eqref{eq:associated-laguerre}, and let $\{ \varphi_{m} \}_{m \in \Z}$ be the Laguerre functions
\begin{equation*}
\varphi_{m}(t) = \sqrt{2} \, \mathrm{L}_m(2 t) e^{-t} \mathbf{1}_{[0,\infty)}(t) \quad \text{ and } \quad  \varphi_{-m-1}(t) = -\sqrt{2} \, \mathrm{L}_m(-2 t) e^{t} \mathbf{1}_{(-\infty,0)}(t)
\end{equation*}
for $m \in \N_0$, where $\mathrm{L}_m = \mathrm{L}_m^{(0)}$ and $\mathbf{1}_A$ denotes the indicator function of a set $A$.
Consider half-integer order $\alpha = \nu + 1/2$ for $\nu \in \N_0$.
Then the Mat\'ern-Laguerre functions
\begin{align*}
\psi^{+}_{m, \nu}(t) &= \frac{\nu!}{(2\nu)!} \frac{m!}{(m+\nu+1)!} (2 t)^{\nu+1} \mathrm{L}_m^{(\nu+1)}(2 t) e^{-t} \mathbf{1}_{[0,\infty)}(t) &&\text{ for } \quad m \in \mathbb{N}_0\\
\psi^{-}_{m, \nu}(t) &= (-1)^\nu \psi^{+}_{m, \nu}(-t)  &&\text{ for } \quad m \in \mathbb{N}_0, \\
\psi^{0}_{m, \nu}(t) &= \frac{1}{ \sqrt{\smash[b]{2}} } \frac{\nu!}{\sqrt{\smash[b]{(2\nu)!}}} \sum_{k=0}^{\nu+1} {\nu+1\choose k}(-1)^k \varphi_{m+k-\nu-1}(t) &&\text{ for } m = 0, \ldots, \nu
\end{align*}
form an orthonormal basis of the RKHS and
\begin{equation*}
r_{\nu + 1/2,}(t, u) = \sum_{m=0}^\nu \psi^{0}_{m,\nu}(t) \psi^{0}_{m,\nu}(u) +  \sum_{m=0}^\infty \psi^{-}_{m,\nu}(t)\psi^{-}_{m,\nu}(u) + \sum_{m=0}^\infty \psi^{+}_{m,\nu}(t)\psi^{+}_{m,\nu}(u)
\end{equation*}
for all $t, u \in \R$.
The basis functions $\smash[b]{\psi_{m,\nu}^{-}}$ and $\smash[b]{\psi_{m, \nu}^{+}}$ are orthogonal in $\mathscr{L}_2(\R, w_{\nu})$ for the weight function $\smash[b]{w_{\nu}(t) = 2 / \abs{2 t}^{\nu+1}}$.

\subsection{Cauchy kernel}

Expansions for the Cauchy kernel are derived in \Cref{sec:cauchy}.
The Cauchy kernel is
\begin{equation*}
  r(t, u) = \frac{1}{1 + (t - u)^2}.
\end{equation*}
Both the complex-valued Cauchy--Laguerre functions
\begin{equation*}
\psi_{m}(t) = -\frac{1}{ \sqrt{2} } \frac{ (i t)^m }{(i t -1)^{m+1}}  \quad \text{ and } \quad \psi_{-m-1}(t) = -\frac{1}{ \sqrt{2} } \frac{ (i t)^m }{(i t + 1)^{m+1}}
\end{equation*}
for $m \in \N_0$ and the \rev{real-valued Cauchy--Laguerre functions}
\begin{equation*}
\alpha_{m}(t) = \frac{1}{\sqrt{2}} \big( \psi_{m}(t) + \psi_{m}^*(t) \big) \quad \text{and} \quad \beta_{m}(t) = \frac{1}{\sqrt{2}} \big( \psi_{m}(t) - \psi_{m}^*(t) \big)
\end{equation*}
for $m \in \N_0$ form orthonormal bases of the RKHS.
Therefore, the Cauchy kernel has the expansions
\begin{equation*}
r(t, u) = \sum_{m=-\infty}^\infty \psi_{m}^*(t)\psi_{m}(u) =  \sum_{m=0}^\infty \alpha_{m}(t)\alpha_{m}(u) + \sum_{m=0}^\infty \beta_{m}(t)\beta_{m}(u)
\end{equation*}
for all $t, u \in \R$.
Expressions of $\alpha_{m}$ and $\beta_{m}$ in terms of real parameters are given in \eqref{eq:real_cauchy_laguerre_explicit}.

\subsection{Gaussian kernel}

Expansions for the Gaussian kernel are derived in \Cref{sec:gaussian}.
The Gaussian kernel is
\begin{equation*}
  r(t, u) = \exp\bigg( \! -\frac{1}{2} (t - u)^2 \bigg).
\end{equation*}
The functions
\begin{equation}
  \label{eq:gaussian-basis-intro}
  \psi_{m}(t) = \bigg( \frac{2 \sqrt{2}}{3} \bigg)^{1/2} \sqrt{ \frac{1}{6^m m!} }  e^{-t^2/3} \mathrm{H}_m\bigg( \frac{2 t}{\sqrt{3}} \bigg) \quad \text{ for } \quad m \in \N_0
\end{equation}
form an orthonormal basis of the RKHS and the kernel has the expansion
\begin{equation*}
  r(t, u) = \sum_{m=0}^\infty \psi_{m}(t) \psi_{m}(u)
\end{equation*}
for all $t, u \in \R$.
This expansion is a special case of the well-known Mercer expansion of the Gaussian kernel~\cite[Section~12.2.1]{FasshauerMcCourt2015}.
The basis functions~\eqref{eq:gaussian-basis-intro} are orthogonal in $\mathscr{L}_2(\R, w_\alpha)$ for the weight function $w_\alpha(t) = \alpha \pi^{-1/2} e^{-\alpha^2 t^2}$ with $\alpha = \sqrt{\smash[b]{2/3}}$.


\section{Expansions of Matérn kernels}
\label{sec:matern}
The Mat\'ern kernel of order $\alpha > 0$ in~\eqref{eq:matern-kernel-intro} can be written as
\begin{equation*}
r_{\alpha}(t, u) =   \frac{2^{1-2\alpha}}{\Gamma(\alpha)} (2 \abs{t-u})^\alpha \mathrm{K}_\alpha( \abs[0]{t-u}  ),
\end{equation*}
and its Fourier transform is~\rev{\citep[e.g.,][Theorem~6.13]{Wendland2005}}
\begin{equation*}
\hat{\Phi}_{\alpha}(\omega) = 2^{1-2\alpha} \sqrt{\pi} \, \frac{\Gamma(\alpha+1/2)}{\Gamma(\alpha)} \frac{2^{2\alpha}}{ (\omega^2 + 1)^{\alpha + 1/2 } }.
\end{equation*}
From now on we assume that the kernel is of half-integer order: $\alpha = \nu + 1/2$ for $\nu \in \N_0$.
Then the Fourier transform simplifies to
\begin{equation*}
\hat{\Phi}_{\nu+1/2}(\omega) = \frac{  (\nu!)^2 }{(2\nu)!} \frac{2^{2\nu+1}}{(\omega^2+1)^{\nu+1}},
\end{equation*}
and a non-symmetric square-root, in the sense that $\abs[0]{\hat{h}_{\nu+1/2}(\omega)}^2 = \hat{\Phi}_{\nu+1/2}(\omega)$, is given by
\begin{equation}
  \label{eq:h-hat-matern}
\hat{h}_{\nu+1/2}(\omega) = \frac{ \nu! }{\sqrt{\smash[b]{(2\nu)!}}} \frac{2^{\nu+1/2}}{(i\omega + 1)^{\nu+1}}.
\end{equation}
The corresponding time domain function is~\cite[Section~1.03]{Wiener1949}
\begin{equation}
\label{eq:h-matern}
  h_{\nu+1/2}(t) = 2^{\nu+1/2} \frac{ \nu! }{\sqrt{\smash[b]{(2\nu)!}}} \frac{t^\nu}{\nu !} e^{-t} \mathbf{1}_{[0,\infty)}(t).
\end{equation}
\rev{Note that} this function vanishes on the negative real line.

\subsection{Laguerre functions}
\label{sec:laguerre-functions}
The following material is mostly based on Section~2.6.4 in~\cite{Higgins1977} and Section~1.03 in~\cite{Wiener1949}.
To derive an orthonormal expansion for the Matérn kernel we use the so-called \emph{Laguerre functions} $\varphi_{m}$ whose Fourier transforms are given by
\begin{equation}
  \label{eq:laguerre-func-ft}
\hat{\varphi}_{m}(\omega) = \sqrt{2} \frac{(i\omega-1)^m}{(i\omega+1)^{m+1}} \quad \text{ for } \quad m \in \mathbb{Z}.
\end{equation}
The functions $\hat{\varphi}_{m}$ form an orthonormal basis of $\mathscr{L}_2(\mathbb{R}, 1/2\pi)$.
Because the Fourier transform is an isometry, the Laguerre functions themselves, defined by the inverse Fourier transform
\begin{equation} \label{eq:laguerre-inverse-ft}
\varphi_{m}(t) = \frac{1}{2\pi} \int_{-\infty}^\infty \hat{\varphi}_{m}(\omega) e^{i\omega t} \dif \omega \quad \text{ for } \quad m \in \mathbb{Z},
\end{equation}
are an orthonormal basis of $\mathscr{L}_2(\R)$.
Let $\mathrm{L}_m$ for $m \in \N_0$ be the $m$th Laguerre polynomial
\begin{equation}
  \label{eq:laguerre-polynomial}
  \mathrm{L}_m(t) = \sum_{k=0}^m \binom{m}{k} \frac{(-1)^k}{k!} t^k.
\end{equation}
For non-negative indices $m \in \N_0$ the inverse Fourier transform~\eqref{eq:laguerre-inverse-ft} is given by
\begin{equation}
  \label{eq:phi-m-matern-non-negative}
\varphi_{m}(t) = \sqrt{2} \, \mathrm{L}_m(2 t) e^{-t} \mathbf{1}_{[0,\infty)}(t).
\end{equation}
The conjugate symmetry $\hat{\varphi}_{-m-1}^*(\omega) = -\hat{\varphi}_{m}(\omega)$ gives the following expression for negative indices:
\begin{equation*}
  \varphi_{-m-1}(t) = - \varphi_{m}(-t) = -\sqrt{2} \, \mathrm{L}_m(-2 t) e^{t} \mathbf{1}_{(-\infty, 0)}(t) \quad \text{ for } \quad  m \in \mathbb{N}_0.
\end{equation*}
The Laguerre functions and their Fourier transforms satisfy the following useful identities:
\begin{align*}
\hat{\varphi}_{-m-1}^*(\omega) &= -\hat{\varphi}_{m}(\omega), \tag{conjugate symmetry} \\
\hat{\varphi}_{m+k}(\omega) &= \Big( \frac{i\omega-1}{i\omega + 1}  \Big)^k \hat{\varphi}_{m}(\omega), \tag{shift property} \\
\hat{\varphi}_{m}(\omega)\hat{\varphi}_{k}(\omega) &= \frac{1}{ \sqrt{2} } \big(  \hat{\varphi}_{m+k}(\omega) - \hat{\varphi}_{m+k+1}(\omega)  \big), \tag{multiplication property} \\
\frac{ 2^{\nu+1/2} }{(i\omega + 1)^{\nu+1} }  &= \sum_{k=0}^\nu {\nu \choose k} (-1)^k \hat{\varphi}_{k}(\omega). \tag{binomial identity} \\
\end{align*}

\subsection{Mat\'ern--Laguerre functions}
\label{sec:matern-laguerre}
In view of \Cref{thm:main-theorem}, an orthonormal basis for the RKHS of the Matérn kernel $r_{\nu, + 1/2}$ is obtained from~\eqref{eq:h-hat-matern} and~\eqref{eq:laguerre-func-ft} in Fourier domain as
\begin{equation}\label{eq:matern_laguerre_functions}
\hat{\psi}_{m,\nu}(\omega) = \hat{h}_{\nu+1/2}(\omega)\hat{\varphi}_{m}(\omega)
= 2^{\nu+1}\frac{ \nu! }{\sqrt{\smash[b]{(2\nu)!}}}  \frac{(i\omega-1)^m}{(i\omega+1)^{m+1+\nu+1}}
\end{equation}
for $m \in \Z$.
We call the resulting functions the \emph{Mat\'ern--Laguerre functions}.
Like the Laguerre functions, the Mat\'ern--Laguerre functions satisfy a certain conjugate symmetry property in the sense that
\begin{equation}\label{eq:matern_laguerre_conjugate_symmetry}
\hat{\psi}_{-\nu-1-m-1,\nu}(\omega) = (-1)^{\nu} \hat{\psi}^*_{m,\nu}(\omega) \quad \text{ for } \quad m \in \mathbb{N}_0.
\end{equation}
Furthermore, by the binomial identity and the shift property of Laguerre functions, the Matérn--Laguerre functions and their Fourier transforms are
\begin{equation}
  \label{eq:matern-laguerre-functions}
  \psi_{m,\nu}(t) = \frac{1}{ \sqrt{2} } \frac{ \nu! }{\sqrt{\smash[b]{(2\nu)!}}} \sum_{k=0}^{\nu+1}  {\nu+1 \choose k}(-1)^k \varphi_{m+k}(t)
\end{equation}
and
\begin{equation} \label{eq:matern_laguerre_binomial_expression}
\hat{\psi}_{m,\nu}(\omega) = \frac{1}{ \sqrt{2} } \frac{ \nu! }{\sqrt{\smash[b]{(2\nu)!}}}   \sum_{k=0}^{\nu+1}  {\nu+1 \choose k}(-1)^k \hat{\varphi}_{m+k}(\omega)
\end{equation}
for $m \in \Z$.
The Matérn kernel of order $\nu + 1/2$ can therefore be expanded as
\begin{equation} \label{eq:matern-expansion}
  r_{\nu + 1/2}(t, u) = \sum_{m=-\infty}^\infty \psi_{m,\nu}(t) \psi_{m,\nu}(u).
\end{equation}
The following proposition provides a uniform upper bound on the Matérn--Laguerre functions.
\begin{proposition}[Matérn--Laguerre upper bound]
  \label{prop:matern_laguerre_bound}
For all $t \in \R$ and $m \in \Z$, the Mat\'ern--Laguerre functions satisfy
\begin{equation*}
\abs[0]{\psi_{m,\nu}(t)} \leq \frac{ 2^\nu \nu! }{\sqrt{\smash[b]{(2\nu)!}}} \sim (\pi \nu)^{1/4} \quad \text{ as } \quad \nu \to \infty.
\end{equation*}
\end{proposition}

\begin{proof}
By~\eqref{eq:matern_laguerre_binomial_expression} and the binomial identity for Laguerre functions,
\begin{equation*}
\psi_{m,\nu}(t) =  \frac{ \nu! }{\sqrt{\smash[b]{(2\nu)!}}}  \sum_{k=0}^\nu {\nu \choose k} (-1)^k  \frac{1}{2\pi}  \int_{-\infty}^\infty \hat{\varphi}_{m}(\omega) \hat{\varphi}_{k}(\omega) e^{i\omega t} \dif \omega.
\end{equation*}
Apply the triangle inequality, the Cauchy--Schwartz inequality, and the orthonormality in $\mathscr{L}_2(\R, 1/2\pi)$ of $\hat{\varphi}_{m}$ to arrive at
\begin{equation*}
\begin{split}
\abs[0]{\psi_{m,\nu}(t)} &\leq \frac{ \nu! }{\sqrt{(2\nu)!}}  \sum_{k=0}^\nu {\nu \choose k}  \frac{1}{2\pi}  \abs{\int_{-\infty}^\infty e^{i\omega t}\hat{\varphi}_{m}(\omega) \hat{\varphi}_{k}(\omega) \dif \omega} \\
&\leq \frac{ \nu! }{\sqrt{\smash[b]{(2\nu)!}}}  \sum_{k=0}^\nu {\nu \choose k} \frac{1}{2\pi} \bigg( \int_{-\infty}^\infty \abs[0]{\hat{\varphi}_{m}(\omega)}^2 \dif \omega \bigg)^{1/2}\bigg( \int_{-\infty}^\infty \abs[0]{\hat{\varphi}_{k}(\omega)}^2 \dif \omega  \bigg)^{1/2} \\
&= \frac{ \nu! }{\sqrt{\smash[b]{(2\nu)!}}}  \sum_{k=0}^\nu {\nu \choose k} \\
&= \frac{ 2^\nu \nu! }{\sqrt{\smash[b]{(2\nu)!}}}.
\end{split}
\end{equation*}
The asymptotic equivalence as $\nu \to \infty$ follows from Stirling's formula.
\end{proof}

It appears difficult to improve upon the bound in \Cref{prop:matern_laguerre_bound}.
Consequently, uniform convergence of Mat\'ern--Laguerre expansions on $\R$ is likely unattainable.

\subsection{Classification of Matérn--Laguerre functions}
\label{sec:matern-classification}

For $m \in \N_0$, a more compact and convenient expression of the Mat\'ern--Laguerre functions~\eqref{eq:matern-laguerre-functions} may be obtained by using the convolution formula in \Cref{thm:main-theorem}.
For $\eta \in \N_0$, the associated Laguerre polynomial $\mathrm{L}_m^{(\eta)}$ is defined as
\begin{equation} \label{eq:associated-laguerre}
  \mathrm{L}_m^{(\eta)}(t) = \sum_{k=0}^m \binom{m + \eta}{m - k} \frac{(-1)^k}{k!} t^k.
\end{equation}
The associated Laguerre polynomial $\mathrm{L}_m^{(0)}$ equals the Laguerre polynomial $\mathrm{L}_m$ in~\eqref{eq:laguerre-polynomial}.
For $t > 0$ and $m \in \mathbb{N}_0$, we get from~\eqref{eq:h-matern} and~\eqref{eq:phi-m-matern-non-negative} that
\begin{equation*}
\begin{split}
  \psi_{m, \nu}(t) &= \int_{-\infty}^\infty h_{\nu+1/2}(t - \tau) \varphi_{m}(\tau) \dif \tau \\
  &= \frac{\nu!}{\sqrt{\smash[b]{(2\nu)!}}} 2^{\nu+1/2} \int_0^t e^{-(t-\tau)}\frac{(t - \tau)^\nu}{\nu!} e^{-\tau} \sqrt{2} \mathrm{L}_m( 2\tau) \dif \tau \\
  &= \frac{\nu!}{\sqrt{\smash[b]{(2\nu)!}}} 2 e^{-t} \int_0^t \frac{(2t - 2\tau)^\nu}{\nu!}\mathrm{L}_m( 2\tau) \dif \tau \\
  &= \frac{\nu!}{\sqrt{\smash[b]{(2\nu)!}}}  e^{-t} \int_0^{2t} \frac{(2t - \tau)^\nu}{\nu!}\mathrm{L}_m( \tau) \dif \tau \\
&= \frac{\nu!}{\sqrt{\smash[b]{(2\nu)!}}} \frac{m!}{(m+\nu+1)!}  (2t)^{\nu+1} \mathrm{L}_m^{(\nu+1)}(2t) e^{-t},
\end{split}
\end{equation*}
where the last equality follows from a convolution identity for Laguerre polynomials~\cite[Chapter~6, Problem~(3)]{Bell2004}.
For $t < 0$, the Laguerre functions $\varphi_{m}(t)$ vanish and the convolution evaluates to zero and hence
\begin{equation*}
\psi_{m, \nu}(t) = \frac{\nu!}{\sqrt{\smash[b]{(2\nu)!}}} \frac{m!}{(m+\nu+1)!}  (2 t)^{\nu+1} \mathrm{L}_m^{(\nu+1)}( 2 t) e^{- t} \mathbf{1}_{[0,\infty)}(t) \quad \text{ for } \quad m \in \mathbb{N}_0.
\end{equation*}
For negative indices $m \leq -\nu - 2$ a similar expression is obtained from the conjugate symmetry~\eqref{eq:matern_laguerre_conjugate_symmetry}:
\begin{equation}\label{eq:negative_matern_laguerre}
  \begin{split}
    \psi_{-\nu-1-m-1,\nu}(t) &= (-1)^{\nu} \psi_{m,\nu}(-t) \\
    &= -\frac{\nu!}{\sqrt{\smash[b]{(2\nu)!}}} \frac{m!}{(m+\nu+1)!}  (2 t)^{\nu+1} \mathrm{L}_m^{(\nu+1)}(2\abs[0]{t}) e^{-\abs[0]{t}} \mathbf{1}_{(-\infty,0)}(t)
  \end{split}
\end{equation}
for $m \in \mathbb{N}_0$.
This motivates the following notation for the three classes of Mat\'ern--Laguerre functions that comprise an orthonormal basis:
\begin{align}
  \psi_{m, \nu}^{+}(t) &= \psi_{m, \nu}(t) &&\text{ for } \quad m \in \N_0, \label{eq:psi-positive} \\
  \psi_{m, \nu}^{-}(t) &= (-1)^{\nu} \psi_{m,\nu}(-t) &&\text{ for } \quad m \in \N_0, \label{eq:psi-negative} \\
  \psi_{m, \nu}^{0}(t) &= \psi_{-\nu-1+m,\nu}(t) &&\text{ for } \quad  m = 0,1, \ldots ,\nu. \label{eq:psi-null}
\end{align}
For convenience, define the corresponding sets
\begin{equation*}
  \mathscr{M}_{\nu}^+ = \big\{ \psi_{m,\nu}^{+} \big\}_{m \in \N_0}, \quad \mathscr{M}_{\nu}^- = \big\{ \psi_{m,\nu}^{-} \big\}_{m \in \N_0}, \quad \mathscr{M}_{\nu}^0 = \big\{ \psi_{m,\nu}^{0} \big\}_{m=0}^\nu,
\end{equation*}
the union $\mathscr{M}_{\nu} = \mathscr{M}_{\nu}^- \cup \mathscr{M}_{\nu}^+$, and the kernels
\begin{equation} \label{eq:rho-kernels}
  \rho_{\nu + 1/2}^-(t, u) = \sum_{m=0}^\infty \psi_{m, \nu}^{-}(t) \psi_{m, \nu}^{-}(u) \quad \text{and} \quad \rho_{\nu + 1/2}^+(t, u) = \sum_{m=0}^\infty \psi_{m, \nu}^{+}(t) \psi_{m, \nu}^{+}(u).
\end{equation}
We call the set $\mathscr{M}_{\nu}^0$ the \emph{null-space} and study it in more detail in \Cref{sec:null-space}.
For now, note that the null-space functions are supported on $\R$ because from~\eqref{eq:matern-laguerre-functions} one can see that for $m = 0, \ldots, \nu$ the sum that defines $\psi_{-\nu-1+m,\nu}$ contains Laguerre functions with both negative and non-negative indices.
Some of the basis functions are shown in \Cref{fig:matern-null-space,fig:matern-non-null-space}.

\begin{figure}[t]
  \centering
  \includegraphics[width=\textwidth]{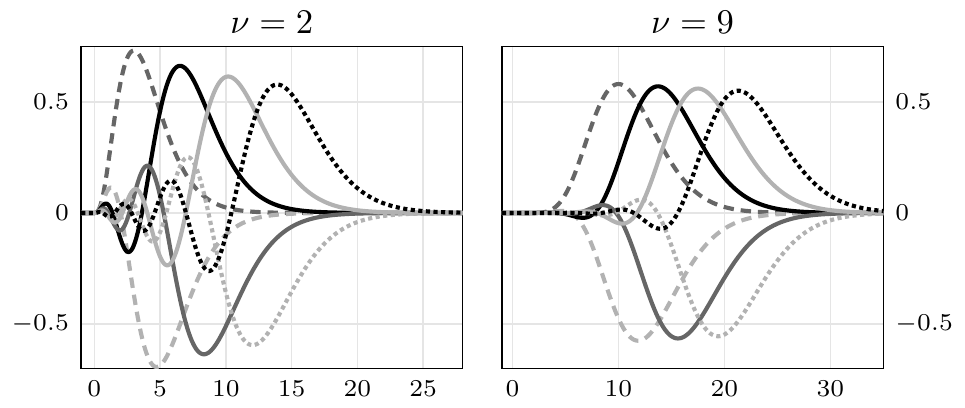}
  \caption{The Matérn--Laguerre functions $\psi_{m, \nu}^{+}$ in~\eqref{eq:psi-positive} for $m = 0, \ldots, 6$. Observe that the functions vanish on the negative real line.}
  \label{fig:matern-non-null-space}
\end{figure}

The Matérn expansion~\eqref{eq:matern-expansion} can now be written in terms of these functions and kernels as
\begin{equation*}
  r_{\nu + 1/2}(t, u) = \sum_{m=0}^\nu \psi_{m, \nu}^{0}(t) \psi_{m, \nu}^{0}(u) + \rho_{\nu + 1/2}^-(t, u) + \rho_{\nu + 1/2}^+(t, u).
\end{equation*}
It is clear that the functions in $\mathscr{M}_{\nu}^-$ are supported on the negative real line and the functions in $\mathscr{M}_{\nu}^+$ on the positive real line.
This observation yields the following simplifications:
\begin{align}
  r_{\nu + 1/2}(t, u) &= \sum_{m=0}^\nu \psi_{m, \nu}^{0}(t) \psi_{m, \nu}^{0}(u) + \rho_{\nu + 1/2}^+(t, u) &&\text{ if } \quad t \geq 0 \: \text{ or } \: u \geq 0, \label{eq:matern-simplification-positive} \\
  r_{\nu + 1/2}(t, u) &= \sum_{m=0}^\nu \psi_{m, \nu}^{0}(t) \psi_{m, \nu}^{0}(u) + \rho_{\nu + 1/2}^-(t, u) &&\text{ if } \quad t \leq 0 \: \text{ or } \: u \leq 0, \\
    r_{\nu + 1/2}(t, u) &= \sum_{m=0}^\nu \psi_{m, \nu}^{0}(t) \psi_{m, \nu}^{0}(u) &&\text{ if } \quad \operatorname{sign} t \neq \operatorname{sign} u. \label{eq:matern-simplification-sign}
\end{align}
We next show that $\mathscr{M}_{\nu}$, $\mathscr{M}_{\nu}^-$, and $\mathscr{M}_{\nu}^+$ form orthogonal bases with respect to the weight function
\begin{equation*}
w_{\nu}(t) =  2 / \abs{2 t}^{\nu+1}.
\end{equation*}
This justifies saying that the expansions we have derived for Matérn kernels are ``almost'' Mercer.

\begin{proposition}[Matérn--Laguerre orthogonality]
  \label{prop:matern_laguerre_L2}
The sets $\mathscr{M}_{\nu}$, $\mathscr{M}_{\nu}^+$, and $\mathscr{M}_{\nu}^-$ form orthogonal bases in $\mathscr{L}_2(\mathbb{R},w_{\nu} )$, $\mathscr{L}_2(\mathbb{R}_+,w_{\nu})$, and $\mathscr{L}_2(\mathbb{R}_-,w_{\nu} )$, respectively.
Furthermore,
\begin{equation*}
\norm[1]{ \psi^{+}_{m,\nu} }_{ \mathscr{L}_2(\mathbb{R},w_{\nu}) }^2 = \norm[1]{ \psi^{-}_{m,\nu} }_{ \mathscr{L}_2(\mathbb{R},w_{\nu}) }^2 = \frac{(\nu!)^2}{(2\nu)!} \frac{m!}{(m+\nu+1)!} \quad \text{ for every } \quad m \in \N_0.
\end{equation*}

\end{proposition}

\begin{proof}
That $\mathscr{M}_{\nu}^+$ forms an orthogonal basis in $\mathscr{L}_2(\mathbb{R}_+,w_{\nu})$ follows from the fact that the functions
\begin{equation}
t^{\nu/2+1/2} \mathrm{L}_m^{(\nu+1)}(t) e^{-t/2} \quad \text{ for } \quad m \in \N_0
\end{equation}
form an orthonormal basis in $\mathscr{L}_2(\mathbb{R}_+)$ \citep[Theorem~5.7.1]{Szego1939}.
Furthermore, the norms of the functions in $\mathscr{M}_{\nu}^+$ are readily computed from the norms of the corresponding Laguerre polynomials:
\begin{equation*}
\begin{split}
\norm[1]{ \psi^{+}_{m,\nu} }_{ \mathscr{L}_2(\mathbb{R},w_{\nu}) }^2
&= \frac{(\nu!)^2}{(2\nu)!} \bigg( \frac{m!}{(m+\nu+1)!} \bigg)^2 \int_0^\infty [ \mathrm{L}_m^{(\nu+1)}(t) ]^2 t^{\nu+1} e^{-t} \dif t \\
&= \frac{(\nu!)^2}{(2\nu)!} \frac{m!}{(m+\nu+1)!}.
\end{split}
\end{equation*}
The statement pertaining to $\mathscr{M}_{\nu}^-$ follows from the symmetry \eqref{eq:negative_matern_laguerre} and the statement pertaining to $\mathscr{M}_{\nu}$ from the fact that $\mathscr{L}_2(\mathbb{R}) = \mathscr{L}_2(\mathbb{R}_-) \oplus \mathscr{L}_2(\mathbb{R}_+)$.
\end{proof}

\rev{
Because they do not decay to zero sufficiently fast at the origin, the functions in $\mathscr{M}_\nu^0$ are not members of $\mathscr{L}_2(\mathbb{R},w_{\nu})$. This will become evident in Section~\ref{sec:null-space}.
}

\subsection{Truncation error}

\begin{figure}[t]
  \centering
  \includegraphics[width=\textwidth]{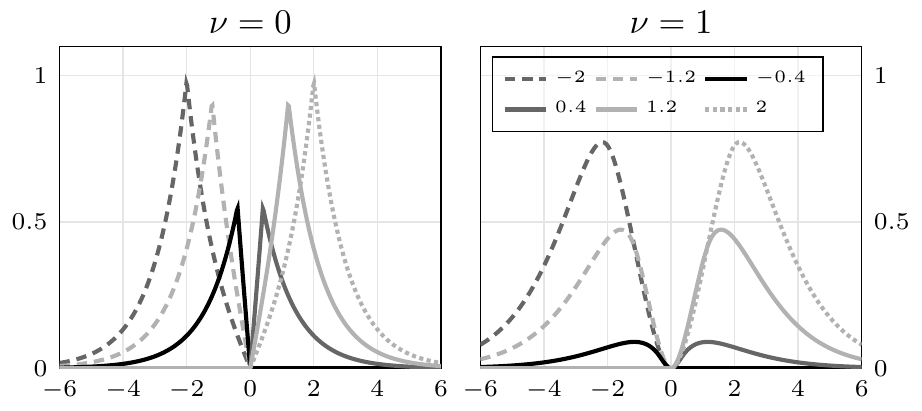}
  \caption{Translates $\rho_{\nu + 1/2}(\cdot, u)$ of the kernel in~\eqref{eq:rho-kernel} for $u \in \{-2, -1.2, -0.4, 0.4, 1.2, 2\}$. Observe that each translate is supported on the axis that $u$ lies on.}
  \label{fig:matern-rho-kernel}
\end{figure}

Define the kernel
\begin{equation} \label{eq:rho-kernel}
  \rho_{\nu + 1/2}(t, u) = \rho_{\nu + 1/2}^-(t, u) + \rho_{\nu + 1/2}^+(t, u)
\end{equation}
in terms of the kernels in~\eqref{eq:rho-kernels}.
A few translates of this kernel are displayed in \Cref{fig:matern-rho-kernel}.
The full Matérn kernel is therefore
\begin{equation*}
r_{\nu+1/2}(t, u) = \sum_{m=0}^\nu \psi_{m, \nu}^{0} (t) \psi_{m, \nu}^{0} (u) + \rho_{\nu + 1/2}(t,u).
\end{equation*}
From \Cref{prop:matern_laguerre_L2} we see that the kernel $\rho_{\nu + 1/2}$ is an element of $\mathscr{L}_2(\mathbb{R}\times \mathbb{R},w_{\nu} \otimes w_{\nu})$ and that its squared norm is given by
\begin{equation*}
  \begin{split}
    \int_{-\infty}^\infty\int_{-\infty}^\infty \rho_{\nu + 1/2}^2&(t,u) w_{\nu}(t)w_{\nu}(u) \dif t\dif u \\
    &= \sum_{m=0}^\infty \big( \norm[1]{ \psi^{-}_{m,\nu}}_{ \mathscr{L}_2(\mathbb{R},w_{\nu}) }^4
+ \norm[1]{ \psi^{+}_{m,\nu}}_{ \mathscr{L}_2(\mathbb{R},w_{\nu}) }^4 \!\big).
  \end{split}
\end{equation*}
This implies that $\rho_{\nu + 1/2}$ defines a Hilbert--Schmidt operator on $\mathscr{L}_2(\mathbb{R},w_{\nu})$ via~\eqref{eq:mercer-integral-operator} and that the above norm is precisely the squared Hilbert--Schmidt norm of this operator~\citep[Chapter~1, \S 1]{Kuo1975}.
Next the approximation errors for appropriately truncated approximations of the Mat\'ern kernel are examined in terms of the Hilbert--Schmidt norm.
Let $n \geq 1$ and define the truncated kernels
\begin{align}
  \rho_{\nu + 1/2, n}(t, u) &= \sum_{m=0}^{n-1} \psi_{m, \nu}^{-}(t) \psi_{m, \nu}^{-}(u) + \sum_{m=0}^{n-1} \psi_{m,\nu}^{+}(t) \psi_{m, \nu}^{+}(u), \\
  r_{\nu + 1/2, n}(t, u) &= \sum_{m=0}^\nu \psi_{m, \nu}^{0} (t) \psi_{m, \nu}^{0} (u)+ \rho_{\nu + 1/2, n}(t, u). \label{eq:matern-truncation}
\end{align}
Observe that $r_{\nu + 1/2,n}$ is a finite expansion of $\nu + 1 + 2n$ terms.
Some truncations of Matérn kernels are displayed in \Cref{fig:matern-truncations}.

\begin{figure}[t]
  \centering
  \includegraphics[width=\textwidth]{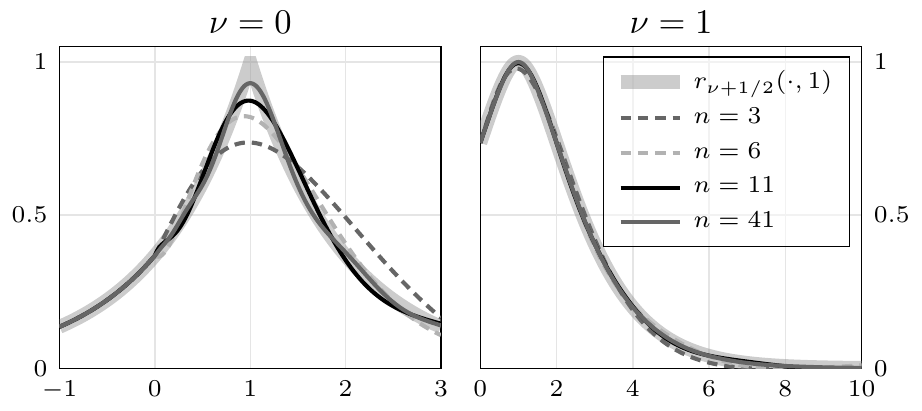}
  \caption{The truncation in~\eqref{eq:matern-truncation} for two Matérn kernels. Because the second kernel argument has been fixed to a positive value, the truncations are exact on the negative real line by~\eqref{eq:matern-simplification-sign}.}
  \label{fig:matern-truncations}
\end{figure}

\begin{proposition}[Matérn truncation]
  \label{prop:matern-truncation-error}
  For every $n \in \N$ it holds that
\begin{equation*}
  \Bigg( \int_{-\infty}^\infty\int_{-\infty}^\infty ( r_{\nu+1/2}(t,u) - r_{\nu+1/2,n}(t,u) )^2 w_{\nu}(t)w_{\nu}(u) \dif t\dif u \Bigg)^{1/2} \leq \frac{c_\nu}{n^{\nu+1/2}},
\end{equation*}
where
\begin{equation*}
  c_\nu = \frac{(\nu!)^2}{(2\nu)!} \sqrt{ \frac{2(2\nu + 2)}{2\nu + 1} } \sim \frac{1 }{2^{2\nu} } \sqrt{ \frac{2\pi(2\nu + 2) \nu}{2\nu + 1} } \quad \text{ as } \quad \nu \to \infty.
\end{equation*}
\end{proposition}
\begin{proof}
Firstly, the truncation error is
\begin{equation*}
  \begin{split}
    r_{\nu+1/2}(t,u) - r_{\nu+1/2,n}(t,u) &= \rho_{\nu + 1/2}(t,u) - \rho_{\nu + 1/2,n}(t,u) \\
    &= \sum_{m=n}^\infty \big[ \psi^{-}_{m,\nu}(t)\psi^{-}_{m,\nu}(u) + \psi^{+}_{m,\nu}(t)\psi^{+}_{m,\nu}(u) \big].
    \end{split}
\end{equation*}
Using \Cref{prop:matern_laguerre_L2}, the squared norm of the truncation error is straight-forwardly computed as
\begin{equation*}
\begin{split}
  \int_{-\infty}^\infty\int_{-\infty}^\infty ( r_{\nu+1/2}&(t,u) - r_{\nu+1/2,n}(t,u) )^2 w_{\nu}(t)w_{\nu}(u) \dif t\dif u \\
  &= \sum_{m=n}^\infty \big( \norm[1]{ \psi^{-}_{m,\nu}}_{ \mathscr{L}_2(\mathbb{R},w_{\nu}) }^4 + \norm[1]{ \psi^{+}_{m,\nu}}_{ \mathscr{L}_2(\mathbb{R},w_{\nu}) }^4 \big) \\
  &= 2 \bigg( \frac{(\nu!)^2}{(2\nu)!} \bigg)^2 \sum_{m=n}^\infty \bigg( \frac{m!}{(m+\nu+1)!} \bigg)^2 \\
  &\leq 2 \bigg( \frac{(\nu!)^2}{(2\nu)!} \bigg)^2 \sum_{m=n}^\infty \frac{1}{m^{2\nu + 2}}.
\end{split}
\end{equation*}
The sum may be estimated with an integral as
\begin{equation*}
  \begin{split}
    \sum_{m=n}^\infty \frac{1}{m^{2\nu + 2}} &\leq \frac{1}{n^{2\nu+2}} + \int_n^\infty \frac{1}{t^{2\nu+2}} \dif t = \frac{1}{n^{2\nu+2}} + \frac{1}{2\nu + 1} \, \frac{1}{n^{2\nu+1}} \leq \frac{2\nu+2}{2\nu+1} \, \frac{1}{n^{2\nu+1}},
\end{split}
\end{equation*}
where $n \geq 1$ was used in the last inequality.
This yields the desired upper bound.
The asymptotic equivalence for $c_\nu$ as $\nu \to \infty$ follows from Stirling's formula.
\end{proof}

\subsection{The null-space $\mathscr{M}_{\nu}^0$} \label{sec:null-space}

\begin{figure}[t]
  \centering
  \includegraphics[width=\textwidth]{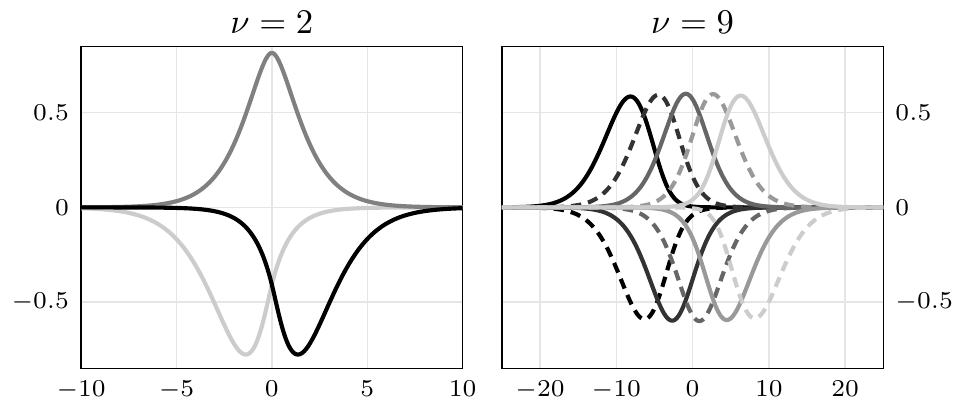}
  \caption{The null-space Matérn--Laguerre functions $\psi_{m, \nu}^{0}$ in~\eqref{eq:psi-null} for $\nu = 2$ and $\nu = 9$.}
  \label{fig:matern-null-space}
\end{figure}

In view of \Cref{prop:matern_laguerre_L2}, $\mathscr{M}_{\nu}^0$ is left as the odd set out.
From~\eqref{eq:h-hat-matern} and \eqref{eq:matern_laguerre_functions} we compute that
\begin{equation*} \label{eq:nullspace_functions}
  \hat{\psi}^{0}_{m,\nu}(\omega) = (-1)^{\nu+1} \hat{h}_{\nu + 1/2}^*(\omega) \hat{\varphi}_{m}(\omega).
\end{equation*}
Furthermore, the functions
\begin{equation*}
(i\omega + 1)^{\nu+1}  \hat{\psi}^{0}_{m,\nu}(\omega),
\end{equation*}
when viewed as functions of $i\omega$, have no poles in the left half-plane.
Therefore $\smash[b]{\mathscr{M}_{\nu}^0}$ are annihilated on the positive real line by the differential operator $\smash[b]{(\mathrm{D} + 1)^{\nu+1}}$.
That is,
\begin{equation*}
(\mathrm{D} + 1)^{\nu+1} \psi_{m,\nu}^{0}(t) = 0 \quad \text{ for every } \quad t > 0.
\end{equation*}
For this reason we refer to these functions as the null-space functions.
The null space functions have a symmetry property similar to that of the functions $\mathscr{M}_{\nu}$ given by~\eqref{eq:psi-positive} and~\eqref{eq:psi-negative}.
\begin{proposition}[Null-space symmetry]
The null-space functions satisfy
\begin{equation*}
\psi^{0}_{\nu-m,\nu}(t) = (-1)^\nu \psi^{0}_{m,\nu}(-t) \quad \text{ and } \quad \hat{\psi}^{0}_{\nu-m,\nu}(\omega) = (-1)^\nu\hat{\psi}^{0*}_{m,\nu} (\omega).
\end{equation*}
for $m = 0, 1, \ldots, \nu$.
\end{proposition}

\begin{proof}
Starting from \eqref{eq:matern_laguerre_binomial_expression}, using the conjugate symmetry of Laguerre functions, and then changing the order of summation gives
\begin{equation*}
\begin{split}
\hat{\psi}^{0}_{m,\nu}(\omega) &= \frac{1}{\sqrt{\smash[b]{2}}} \frac{\nu!}{\sqrt{\smash[b]{(2\nu)!}}} \sum_{k=0}^{\nu+1} {\nu+1\choose k} (-1)^k \hat{\varphi}_{-\nu-1+m+k}(\omega) \\
&= \frac{1}{\sqrt{\smash[b]{2}}} \frac{\nu!}{\sqrt{\smash[b]{(2\nu)!}}} \sum_{k=0}^{\nu+1} {\nu+1\choose k} (-1)^k \hat{\varphi}_{-(\nu-m-k)-1}(\omega) \\
&= -\frac{1}{\sqrt{\smash[b]{2}}} \frac{\nu!}{\sqrt{\smash[b]{(2\nu)!}}} \sum_{k=0}^{\nu+1} {\nu+1\choose k} (-1)^k \hat{\varphi}_{\nu-m-k}^*(\omega)\\
&= -\frac{1}{\sqrt{\smash[b]{2}}} \frac{\nu!}{\sqrt{\smash[b]{(2\nu)!}}} \sum_{k=0}^{\nu+1} {\nu+1\choose \nu+1-k} (-1)^{\nu+1-k} \hat{\varphi}_{\nu-m-(\nu+1-k)}^*(\omega) \\
&= (-1)^\nu\frac{1}{\sqrt{\smash[b]{2}}} \frac{\nu!}{\sqrt{\smash[b]{(2\nu)!}}} \sum_{k=0}^{\nu+1} {\nu+1\choose k} (-1)^{k} \hat{\varphi}_{-\nu-1+\nu-m+k}^*(\omega) \\
&= (-1)^\nu  \hat{\psi}^{0*}_{\nu-m,\nu} (\omega),
\end{split}
\end{equation*}
which is the Fourier domain symmetry.
The time domain symmetry is then obtained from Fourier inversion.
\end{proof}

\begin{example}[Null-space functions]
  \label{example:null-space}
The set $\mathscr{M}_{0}^0$ (i.e., $\nu = 0$) consists of the function
\begin{equation*}
\psi^{0,(0)}_{0}(t) = -e^{- \abs[0]{t}}.
\end{equation*}
The set $\mathscr{M}_{1}^0$ (i.e., $\nu = 1$) consists of the functions
\begin{equation*}
\psi^{0}_{0,1}(t) = \frac{1}{\sqrt{\smash[b]{2}}} (   2 t e^{ t} \mathbf{1}_{(-\infty,0)}(t) + e^{- \abs[0]{t}}  ) \quad \text{ and } \quad \psi^{0}_{1,1}(t) = -\frac{1}{\sqrt{\smash[b]{2}}} ( 2 t e^{- t} \mathbf{1}_{[0,\infty)}(t) + e^{- \abs[0]{t}} ).
\end{equation*}
The set $\mathscr{M}_{2}^0$ (i.e., $\nu = 2$) consists of the functions
\begin{align*}
\psi^{0}_{0,2}(t) &=  \frac{2}{ \sqrt{\smash[b]{4!}}} \big( 2(- t^2 + t) e^{t} \mathbf{1}_{(-\infty,0)}(t) - e^{-\abs[0]{t}} \big), \\
\psi^{0}_{1,2}(t) &= \frac{4}{ \sqrt{\smash[b]{4!}}} ( \abs[0]{t} + 1 )e^{-\abs[0]{t}}, \\
\psi^{0}_{2,2}(t) &= \frac{2}{ \sqrt{\smash[b]{4!}}} \big( -2(t^2 + t )e^{-t} \mathbf{1}_{[0,\infty)}(t) - e^{-\abs[0]{t}} \big).
\end{align*}
\end{example}

Some null space functions are depicted in \Cref{fig:matern-null-space}.
Unlike the basis functions $\mathscr{M}_{\nu}^+$ depicted in \Cref{fig:matern-non-null-space}, the null space functions are supported on the entire real line.
For $d = \abs[0]{t - u}$, a Matérn kernel can be written as
\begin{equation*}
  r_{\nu + 1/2}(t, u) = r_{\nu + 1/2}(0, d) = \sum_{m=0}^\nu \psi_{m, \nu}^{0}(0) \psi_{m, \nu}^{0}(d),
\end{equation*}
where we have used~\eqref{eq:matern-simplification-positive} and the fact that the kernel $\rho_{\nu + 1/2}^+(t, u)$ vanishes if $t = 0$ or $u=0$.
Upon substitution of the expressions in \Cref{example:null-space} we obtain the well-known explicit forms of Matérn kernels in terms of $d$, such as
\begin{equation*}
  r_{3/2}(t, u) = ( 1 + d ) e^{- d} \quad \text{ and } \quad r_{5/2}(t, u) = \bigg(1 + d + \frac{d^2}{3} \bigg) e^{- d}.
\end{equation*}
%


\section{Expansions of the Cauchy kernel}
\label{sec:cauchy}

The Cauchy kernel and its Fourier transform are
\begin{equation} \label{eq:cauchy-kernel}
r(t,u) = \frac{1}{1  + (t-u)^2 } \quad \text{ and } \quad \hat{\Phi}(\omega) = \pi e^{-\abs[0]{\omega}}.
\end{equation}
The Cauchy kernel is thus a Fourier dual to the Mat\'ern kernel of smoothness index $\alpha = 1/2$ (i.e., $\nu = 0$).
In what follows this will inform the construction of an RKHS basis.
A square-root of $\hat{\Phi}(\omega)$ is then given by
\begin{equation} \label{eq:h-cauchy}
\hat{h}(\omega) = \hat{\Phi}(\omega)^{1/2} = \sqrt{\pi} \, e^{ - \abs[0]{\omega} / 2}.
\end{equation}

\subsection{Expansion in complex-valued Cauchy--Laguerre functions}
\label{sec:cauchy-complex}
In view of the Fourier dualism with the Mat\'ern-$\frac{1}{2}$ kernel and the fact that the Fourier transform is an isometry from $\mathscr{L}_2(\R)$ to $\mathscr{L}_2(\R, 1/2\pi)$,
a straight-forward way to construct a suitable basis of $\mathscr{L}_2(\R)$ for Theorem~\ref{thm:main-theorem} is to modify the Laguerre functions from \Cref{sec:laguerre-functions} and consider the functions $\sqrt{\pi}\varphi_{m}(\omega/2)$.
The Fourier transforms of these functions are an orthonormal basis of $\mathscr{L}_2(\R)$, so that Theorem~\ref{thm:main-theorem} and~\eqref{eq:h-cauchy} yield the RKHS basis functions
\begin{equation*}
\hat{\psi}_{m}(\omega) = \sqrt{\pi} \, e^{ - \abs[0]{\omega} / 2 } \sqrt{\pi} \, \varphi_{m}(\omega/2)
\end{equation*}
in the Fourier domain.
Since their inverse Fourier transforms are complex-valued, we call these functions the \emph{complex-valued Cauchy--Laguerre functions}.
For $m \in \N_0$, Fourier inversion gives
\begin{equation*}
\begin{split}
  \psi_{m}(t) = \frac{1}{2} \int_{-\infty}^\infty e^{-\abs[0]{\omega} / 2} \varphi_{m}(\omega/2) e^{i\omega t} \dif \omega
&= \int_{-\infty}^\infty e^{-\abs[0]{\omega}} \varphi_{m}(\omega) e^{i\omega 2t} \dif \omega  \\
&=  \int_0^\infty e^{-\abs[0]{\omega}} \varphi_{m}(\omega) e^{i\omega 2t} \dif \omega \\
&=   \int_0^\infty e^{-\abs[0]{\omega}} \varphi_{m}(\omega) e^{-i\omega ( -2t - i )} \dif \omega \\
&= \hat{\varphi}_m(-2t-i) \\
& = -\frac{1}{\sqrt{2}} \frac{( it )^m}{(it - 1)^{m+1}}.
\end{split}
\end{equation*}
Similarly, for negative indices we get
\begin{equation*}
\begin{split}
  \psi_{-m}(t) = \frac{1}{2} \int_{-\infty}^\infty e^{-\abs[0]{\omega} / 2} \varphi_{-m}(\omega/2) e^{i\omega t} \dif \omega
&=  \int_{-\infty}^\infty e^{-\abs[0]{\omega} } \varphi_{-m}(\omega) e^{i\omega 2t} \dif \omega \\
&= -\int_{-\infty}^0 e^{-\abs[0]{\omega} } \varphi_{m-1}(-\omega) e^{i\omega 2t} \dif \omega \\
&= -\int_0^\infty e^{-\abs[0]{\omega} } \varphi_{m-1}(\omega) e^{-i\omega 2t} \dif \omega \\
&= -\int_0^\infty \varphi_{m-1}(\omega) e^{-i\omega (2t  -i)} \dif \omega \\
&=  -\hat{\varphi}_{m-1}(2t - i) \\
&= -\frac{1}{\sqrt{2}} \frac{(it)^{m-1}}{\big( it + 1 \big)^m}.
\end{split}
\end{equation*}
To summarise, the complex valued Cauchy--Laguerre functions are
\begin{subequations} \label{eq:cauch-laguerre-funcs}
\begin{align}
  \psi_{m}(t) &= - \frac{1}{\sqrt{2}} \frac{ (it)^m }{(it-1)^{m+1}} \quad \text{ for } \quad m \in \mathbb{N}_0, \\
  \psi_{-m-1}(t) &= -\frac{1}{\sqrt{2}} \frac{ (it)^m }{(it+1)^{m+1} } \quad \text{ for } \quad m \in \mathbb{N}_0.
\end{align}
\end{subequations}
They have the conjugate symmetry property
\begin{equation*}
\psi_{m}^*(t) = -\psi_{-m-1}(t) = \psi_{m}(-t) \quad \text{ for } \quad m \in \Z.
\end{equation*}
An expansion of the Cauchy kernel~\eqref{eq:cauchy-kernel} in terms of complex-valued Cauchy--Laguerre functions is thus given by
\begin{equation*}
r(t,u) = \sum_{m=-\infty}^\infty \psi_{m}^*(t)\psi_{m}(u).
\end{equation*}
This expansion is remarkably easy to verify by independent means since geometric summation and conjugate symmetry yield
\begin{equation*}
\sum_{m=0}^\infty \psi_{m}^*(t)\psi_{m}(u) = \frac{1}{2} \frac{1}{ (it - 1)(-iu - 1) - tu  }
\end{equation*}
and
\begin{equation*}
  \sum_{m=-\infty}^{-1} \psi_{m}^*(t)\psi_{m}(u) = \Bigg( \sum_{m=0}^\infty \psi_{m}^*(t)\psi_{m}(u) \Bigg)^*.
\end{equation*}
Hence
\begin{equation*}
\begin{split}
\sum_{m=-\infty}^\infty \psi_{m}^*(t)\psi_{m}(u) = \frac{1}{2} \frac{1}{ 1 - i(t-u)   } + \frac{1}{2} \frac{1}{1 + i(t-u)  } = \frac{1}{1 + (t-u)^2},
\end{split}
\end{equation*}
which indeed is the Cauchy kernel.
An appropriate $\mathscr{L}_2(\R, w)$ space in which the complex-valued Cauchy--Laguerre functions form a complete orthogonal set remains elusive to us.
However, just as with the Mat\'ern--Laguerre expansions in \Cref{sec:matern}, the present expansion is very good at origin since all but two terms vanish:
\begin{equation*}
r(t,0) = \sum_{m=-\infty}^\infty \psi_{m}^*(t)\psi_{m}(0) = - \big( \psi_{-1}(t) + \psi_{0}(t) \big).
\end{equation*}

\subsection{Expansion in \rev{real-valued Cauchy--Laguerre functions}}
\label{sec:cauchy-real}

\begin{figure}[t]
  \centering
  \includegraphics[width=\textwidth]{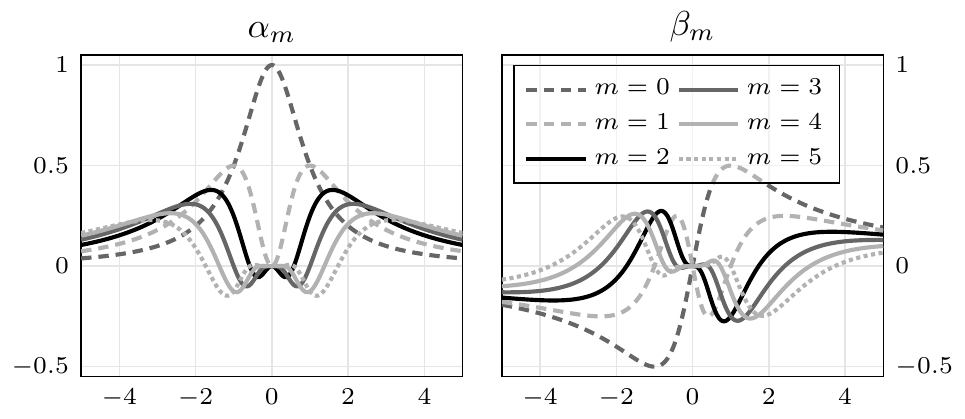}
  \caption{The \rev{real-valued Cauchy--Laguerre functions} $\alpha_{m}$ and $\beta_{m}$ in~\eqref{eq:trig-cauchy-laguerre}.}
  \label{fig:cauchy-basis}
\end{figure}

It would be desirable to obtain a real-valued basis for the Cauchy RKHS.
This can be done by scaling the real and imaginary parts of $\hat{\psi}_{m}$
in a similar manner as was done for the Laguerre functions in~\cite{Christov1982}.
This gives the RKHS basis functions
\begin{equation} \label{eq:trig-cauchy-laguerre}
\alpha_{m}(t) =  \frac{1}{\sqrt{2}} \big( \psi_{m}(t) +  \psi_{m}^*(t)  \big) \quad \text{ and } \quad \beta_{m}(t) =  \frac{1}{i\sqrt{2}} \big( \psi_{m}(t) -  \psi_{m}^*(t)  \big)
\end{equation}
for $m \in \N_0$, where $\psi_{m}$ are the complex-valued Cauchy--Laguerre functions in~\eqref{eq:cauch-laguerre-funcs}.
We call the functions $\alpha_{m}$ and $\beta_{m}$ the \emph{\rev{real-valued Cauchy--Laguerre functions}}.
The binomial theorem yields the explicit expressions
\begin{align*}
\alpha_{m}(t) &= \frac{1}{2} \frac{(-1)^m(it)^m}{ (t^2+1)^{m+1}} \sum_{k=0}^{m+1} {m+1\choose k} (it)^k \big( 1 - (-1)^{m+1-k}  \big),\\
\beta_{m}(t) &= \frac{1}{i2} \frac{(-1)^m(it)^m}{ (t^2+1)^{m+1}} \sum_{k=0}^{m+1} {m+1\choose k} (it)^k \big( 1 + (-1)^{m+1-k}  \big),
\end{align*}
which can be transformed into expressions of only real parameters by considering even and odd $m$ separately.
This yields
\begin{subequations}\label{eq:real_cauchy_laguerre_explicit}
\begin{align}
\alpha_{2m}(t) &= \frac{ (-1)^m t^{2m}}{ (t^2+1)^{2m+1}} \sum_{k=0}^{m} {2m+1\choose 2k} (-1)^k t^{2k} , \\
\alpha_{2m+1}(t) &= \frac{ (-1)^m t^{2m+1}}{ (t^2+1)^{2m+2}} \sum_{k=0}^{m} {2m+2\choose 2k+1} (-1)^k t^{2k+1}, \\
\beta_{2m}(t) &= \frac{ (-1)^m t^{2m}}{ (t^2+1)^{2m+1}} \sum_{k=0}^{m} {2m+1\choose 2k+1} (-1)^k t^{2k+1} , \\
\beta_{2m+1}(t) &= \frac{ (-1)^{m+1} t^{2m+1}}{ (t^2+1)^{2m+2}} \sum_{k=0}^{m+1} {2m+2\choose 2k} (-1)^k t^{2k} .
\end{align}
\end{subequations}
An expansion of the Cauchy kernel~\eqref{eq:cauchy-kernel} in terms of real functions is thus given by
\begin{equation} \label{eq:cauchy-expansion-trig}
r(t,u) = \sum_{m=0}^\infty \alpha_{m}(t)\alpha_{m}(u) + \sum_{m=0}^\infty \beta_{m}(t)\beta_{m}(u).
\end{equation}
At the origin, this reduces to the finite term expansion
\begin{equation*}
  r(t,0) = \alpha_{0}(t)\alpha_{0}(0).
\end{equation*}
The basis functions $\alpha_{m}$ and $\beta_{m}$ and truncations of the expansion~\eqref{eq:cauchy-expansion-trig} are displayed in \Cref{fig:cauchy-basis,fig:cauchy-truncation}.

\begin{figure}[t]
  \centering
  \includegraphics[width=\textwidth]{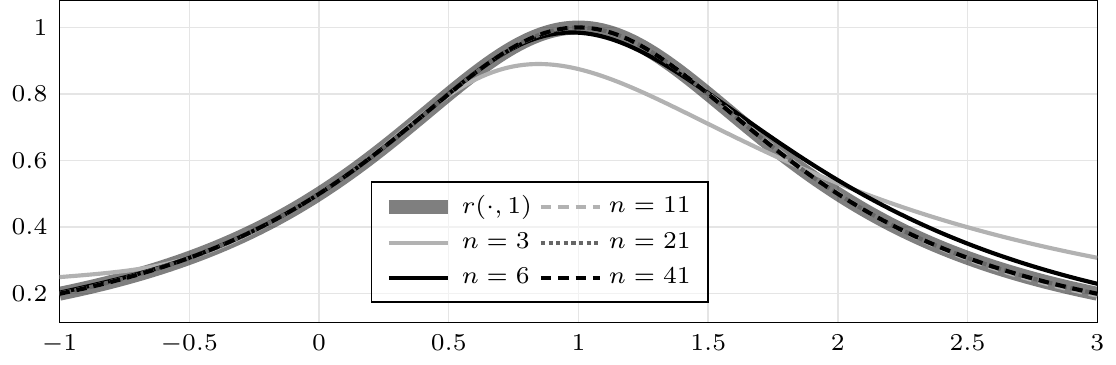}
  \caption{Truncations $\sum_{m=0}^{n-1} \alpha_{m}(t)\alpha_{m}(u) + \sum_{m=0}^{n-1} \beta_{m}(t)\beta_{m}(u)$ of the Cauchy expansion in~\eqref{eq:cauchy-expansion-trig}.}
  \label{fig:cauchy-truncation}
\end{figure}


\section{Expansion of the Gaussian kernel}
\label{sec:gaussian}

The Gaussian kernel and its Fourier transform are
\begin{equation}
  \label{eq:gaussian-kernel}
  r(t, u) = \exp\bigg( \! -\frac{1}{2} (t - u)^2 \bigg) \quad \text{ and } \quad \hat{\Phi}(\omega) = \sqrt{2\pi} \, e^{-\omega^2/2}.
\end{equation}
A square-root is
\begin{equation*}
  \hat{h}(\omega) = \hat{\Phi}(\omega)^{1/2} = (2\pi)^{1/4} e^{-\omega^2/4},
\end{equation*}
so that taking the inverse Fourier transform gives the function $h$ in \Cref{thm:main-theorem} as
\begin{equation} \label{eq:h-gaussian}
  h(t) = 2^{1/4} \pi^{-1/4} e^{-t^2}.
\end{equation}

\subsection{Expansion for the Gaussian kernel}
\label{sec:gaussian-expansion}

As an orthonormal basis of $\mathscr{L}_2(\R)$ we use the \emph{Hermite functions} (for them being an orthonormal basis, see~\cite[Theorem~5.7.1]{Szego1939})
\begin{equation}
  \label{eq:hermite-function}
  \varphi_{m}(t) = \sqrt{ \frac{1}{2^m m! \sqrt{\pi}} } \, e^{-t^2/2} \mathrm{H}_m(t ) \quad \text{ for } \quad m \in \N_0.
\end{equation}
Here $\mathrm{H}_m$ is the $m$th physicist's Hermite polynomial given by
\begin{equation}
  \label{eq:hermite-polynomial}
  \mathrm{H}_m(t) = m! \sum_{k=0}^{\lfloor m / 2 \rfloor} \frac{(-1)^k}{k! (m - 2k)!} (2 t)^{m - 2k}.
\end{equation}
By \Cref{thm:main-theorem}, the functions
\begin{equation*}
    \psi_{m}(t) = \int_{-\infty}^\infty h(t - \tau) \varphi_{m}(\tau) \dif \tau = \bigg( \frac{\sqrt{2} }{\pi} \bigg)^{1/2} \sqrt{ \frac{1}{2^m m! } } \int_{-\infty}^\infty e^{-(t-\tau)^2} e^{- \tau^2/2} \mathrm{H}_m( \tau ) \dif \tau
\end{equation*}
form an orthonormal basis of the RKHS of the Gaussian kernel~\eqref{eq:gaussian-kernel}.
Equation~(17) in Section~16.5 of~\cite{Erdelyi1954} states that
\begin{equation*}
  \int_{-\infty}^\infty e^{-(s - \tau)^2} \mathrm{H}_m(a \tau) \dif \tau = \sqrt{\pi} (1 - a^2)^{m/2} \mathrm{H}_m\bigg( \frac{a s}{\sqrt{1-a^2}} \bigg)
\end{equation*}
for any reals $s$ and $a$.
Completing the square, doing a change of variables, and using this equation yields
\begin{equation*}
  \begin{split}
    \int_{-\infty}^\infty e^{-(t-\tau)^2} e^{-\tau^2/2} \mathrm{H}_m( \tau ) \dif \tau &= \sqrt{\frac{2}{3}} e^{- t^2 / 3} \int_{-\infty}^\infty e^{-(\sqrt{\smash[b]{2/3}} \, t - \tau)^2} \mathrm{H}_m \big( \sqrt{\smash[b]{2/3}} \, \tau \big) \dif \tau \\
    &= \sqrt{ \frac{2\pi}{3} } 3^{-m/2} e^{-t^2/3} \mathrm{H}_m\bigg( \frac{2t}{\sqrt{3}} \bigg).
    \end{split}
\end{equation*}
We thus obtain the basis functions
\begin{equation}
  \label{eq:gaussian-basis}
  \psi_{m}(t) = \bigg( \frac{2 \sqrt{2}}{3} \bigg)^{1/2} \sqrt{ \frac{1}{6^m m!} }  e^{-t^2/3} \mathrm{H}_m\bigg( \frac{2 t}{\sqrt{3}} \bigg) \quad \text{ for } \quad m \in \N_0
\end{equation}
and the resulting expansion
\begin{equation}
  \label{eq:gaussian-expansion}
  r(t, u) = \sum_{m=0}^\infty \psi_{m}(t) \psi_{m}(u)
\end{equation}
of the Gaussian kernel in~\eqref{eq:gaussian-kernel}.
\Cref{fig:gaussian-basis} displays some of the basis functions.

\begin{figure}[t]
  \centering
  \includegraphics[width=\textwidth]{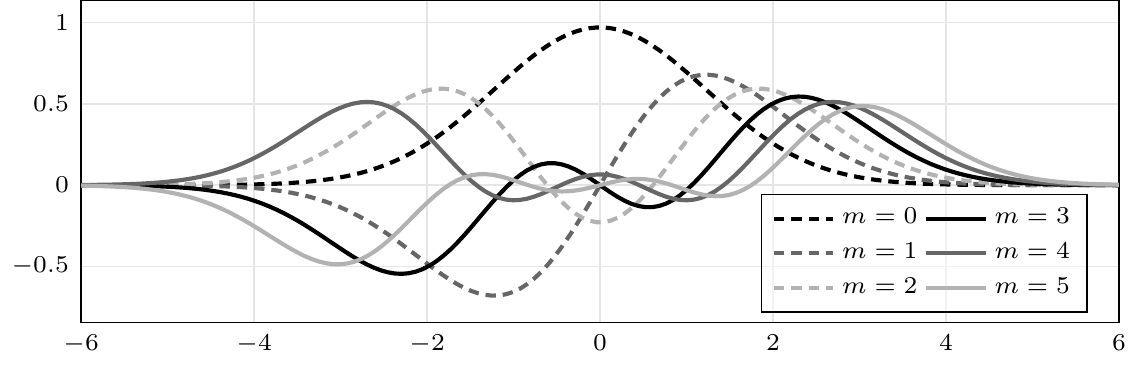}
  \caption{The first six basis functions $\psi_{m}$ in~\eqref{eq:gaussian-basis} of the Gaussian kernel~\eqref{eq:gaussian-kernel}.}
  \label{fig:gaussian-basis}
\end{figure}

Note that the basis functions can be written in terms of the Hermite functions~\eqref{eq:hermite-function} by using the multiplication theorem
\begin{equation*}
  \mathrm{H}_m( b t ) = \sum_{k=0}^{\lfloor m / 2 \rfloor} b^{m - 2k} (b^2 - 1)^k \binom{m}{2k} \frac{(2k)!}{k!} \mathrm{H}_{m - 2k}(t)
\end{equation*}
for Hermite polynomials.
Setting $b = \sqrt{2}$ gives
\begin{equation*}
  \mathrm{H}_m\bigg( \frac{2 t}{\sqrt{3}} \bigg) = 2^{m/2} \sum_{k=0}^{\lfloor m / 2 \rfloor} 2^{-k} \binom{m}{2k} \frac{(2k)!}{k!} \mathrm{H}_{m - 2k}\bigg( \frac{\sqrt{2} \, t}{\sqrt{3}} \bigg),
\end{equation*}
so that
\begin{equation*}
  \psi_{m}(t) = \bigg( \frac{2\sqrt{\pi}}{\sqrt{3}} \bigg)^{1/2} \sqrt{ \frac{2^m m!}{3^m}} \sum_{k=0}^{\lfloor m / 2 \rfloor} \frac{1}{4^k k! \sqrt{(m-2k)!}} \, \varphi_{m-2k}\bigg( \frac{\sqrt{2} \, t}{\sqrt{3}} \bigg).
\end{equation*}
It would be interesting to be able to connect $\psi_{m}$ to the associated Hermite polynomials~\cite{AskeyWimp1984} like the Matérn--Laguerre functions are connected to associated Laguerre functions in \Cref{sec:matern-classification}.

\begin{remark}
Observe that (both here and elsewhere) we have used a basis of $\mathscr{L}_2(\R)$ that is ``compatible'' with the kernel, having the same scaling in the exponential.
That is, the Hermite functions in~\eqref{eq:hermite-function} have the exponential term $e^{-t^2/2}$ and the kernel is $e^{-(t-u)^2/2}$.
For any $\kappa \in (0, \sqrt{2})$, the scaled Hermite functions
\begin{equation*}
        \varphi_{m, \kappa}(t) = \sqrt{ \frac{\kappa}{2^m m! \sqrt{\pi}} } \, e^{-\kappa^2 t^2/2} \mathrm{H}_m( \kappa t )
\end{equation*}
would yield the RKHS basis functions
\begin{equation*}
        \psi_{m, \kappa}(t) = \bigg( \frac{\sqrt{2} \kappa}{a^2} \bigg)^{1/2} \sqrt{ \frac{1}{2^m m!} } \bigg( 1 - \frac{\kappa^2}{a^2} \bigg)^{m/2} e^{-(1 - 1/a^2) t^2} \mathrm{H}_m\bigg( \frac{\kappa t }{a^2 \sqrt{\smash[b]{1 - \kappa^2/a^2}}} \bigg),
\end{equation*}
where $a^2 = 1 + \kappa^2/2$.
\end{remark}

\subsection{Mercer basis and Mehler's formula}

The expansion~\eqref{eq:gaussian-expansion} that we derived for the Gaussian kernel by the use of the basis functions in~\eqref{eq:gaussian-basis} can also be derived by setting
\begin{equation*}
  \rho = \frac{1}{3}, \quad x = \frac{2 t}{\sqrt{3}}, \quad \text{ and } \quad y = \frac{2 u}{\sqrt{3}}
\end{equation*}
in Mehler's formula
\begin{equation*}
  \sum_{m = 0}^\infty \frac{(\rho/2)^m}{m!} \mathrm{H}_m(x) \mathrm{H}_m(y) e^{-(x^2+y^2)/2} = \sqrt{\frac{1}{1-\rho^2}} \exp\bigg( \frac{4xy\rho - (1+\rho^2)(x^2 + y^2)}{2(1-\rho^2)} \bigg)
\end{equation*}
and subsequently multiplying both sides by $e^{-(t^2 + u^2)/3}$.
This suggests that the expansion derived in the preceding section is a special case of the relatively well known Mercer expansion of the Gaussian kernel, \rev{which can also be derived} from Mehler's formula~\cite[Section~12.2.1]{FasshauerMcCourt2015}.
Let $\alpha > 0$ and define the constants
\begin{equation*}
  \beta = \bigg(1 + \frac{2}{\alpha^2} \bigg)^{1/4} \quad \text{ and } \quad \delta^2 = \frac{\alpha^2}{2} ( \beta^2 - 1).
\end{equation*}
The Mercer expansion of the Gaussian kernel with respect to the weight function
\begin{equation*}
  w_\alpha(t) = \frac{\alpha}{\sqrt{\pi}} e^{-\alpha^2 t^2}
\end{equation*}
on the real line is
\begin{equation}
  \label{eq:gaussian-mercer-expansion}
  r(t, u) = \sum_{m=0}^\infty \mu_{m, \alpha} \rev{\vartheta}_{m, \alpha}(t) \rev{\vartheta}_{m, \alpha}(u),
\end{equation}
where
\begin{equation*}
  \mu_{m, \alpha} = \sqrt{ \frac{\alpha^2}{\alpha^2 + \delta^2 + 1/2} } \bigg( \frac{1/2}{\alpha^2 + \delta^2 + 1/2} \bigg)^m
\end{equation*}
are the eigenvalues and
\begin{equation}
  \label{eq:gaussian-basis-mercer}
  \rev{\vartheta}_{m, \alpha}(t) = \sqrt{ \frac{\beta}{2^m m!} } e^{-\delta^2 t^2} \mathrm{H}_m( \alpha \beta t)
\end{equation}
the $\mathscr{L}_2(\R, w_\alpha)$-orthonormal eigenfunctions of the integral operator in~\eqref{eq:mercer-integral-operator}.
By requiring that $\alpha\beta = 2/\sqrt{3}$, so that the Hermite polynomials appearing in~\eqref{eq:gaussian-basis} and~\eqref{eq:gaussian-basis-mercer} have the same scaling, it is straight-forward to solve that
\begin{equation*}
  \psi_{m} = \sqrt{ \mu_{m,\alpha} } \, \rev{\vartheta}_{m,\alpha} = \sqrt{\frac{2}{3^{m+1}}} \, \rev{\vartheta}_{m, \alpha} \quad \text{ when } \quad \alpha = \sqrt{\frac{2}{3}},
\end{equation*}
which shows that the basis~\eqref{eq:gaussian-basis} is a special case of the Mercer basis.
Results of some of the above computations are collected in the following proposition.

\begin{proposition}[Orthogonality of the Gaussian basis]
  \label{prop:gaussian-basis-L2}
  Let $\alpha = \sqrt{\smash[b]{2/3}}$.
  The functions
  \begin{equation*}
    \sqrt{\frac{3^{m+1}}{2}} \, \psi_{m}\rev{(t)} = 2^{1/4} \sqrt{ \frac{1}{2^m m!} }  e^{-t^2/3} \mathrm{H}_m\bigg( \frac{2t}{\sqrt{3}} \bigg) \quad \text{ for } \quad m \in \N_0
  \end{equation*}
  form an orthonormal basis of $\mathscr{L}_2(\R, w_\alpha)$.
\end{proposition}

Although the Mercer expansion~\eqref{eq:gaussian-mercer-expansion} has been known for some time, apparently originating in~\cite[Section~4]{Zhu1998}, all its derivations in the literature that we are aware of are based on Mehler's formula and integral identities for Hermite polynomials (the only detailed derivations that we know of are given in~\cite[Section~12.2.1]{FasshauerMcCourt2015} and~\cite[Section~5.1]{Gnewuch2022}).
The expansion~\eqref{eq:gaussian-expansion} is therefore the first Mercer expansion for the Gaussian kernel that has been derived from some general principle, which in this case is \Cref{thm:main-theorem}, instead of utilising \emph{ad hoc} calculations.
The relative simplicity of the basis functions~\eqref{eq:gaussian-basis} and the fact that the Hermite functions~\eqref{eq:hermite-function} have the same exponential decay as the kernel suggest that the choice $\alpha = \sqrt{\smash[b]{2/3}}$ for the standard deviation of the Gaussian weight $w_\alpha$ may be in some sense the most natural one.
More discussion on the selection of $\alpha$ may be found in~\cite[Section~5.3]{FasshauerMcCourt2012}.

\subsection{Truncation error}

\begin{figure}[t]
  \centering
  \includegraphics[width=\textwidth]{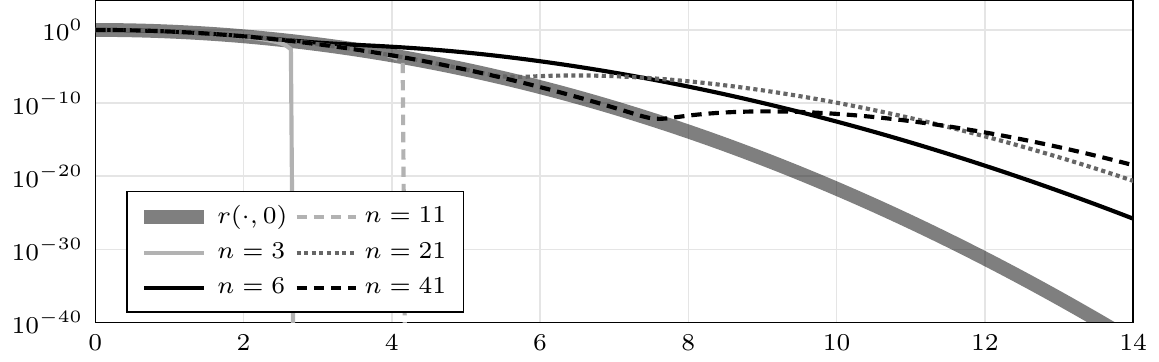}
  \caption{The Gaussian kernel~\eqref{eq:gaussian-kernel} with $u = 0$ and its truncated expansions in~\eqref{eq:gaussian-kernel-truncated}. For $n = 3$ and $n = 11$ the truncated kernels become negative.}
  \label{fig:gaussian-truncation}
\end{figure}

Define the truncated kernel
\begin{equation}
  \label{eq:gaussian-kernel-truncated}
  r_{n}(t, u) = \sum_{m=0}^{n-1} \psi_{m}(t) \psi_{m}(u)
\end{equation}
for any $n \in \N$.
A few truncations are shown in \Cref{fig:gaussian-truncation}.
The truncated kernel converges to the full Gaussian kernel $r$ pointwise on $\R \times \R$.
The following proposition shows that the convergence of~\eqref{eq:gaussian-kernel-truncated} to $r$ is exponential in $\mathscr{L}_2(\R \times \R, w_\alpha \otimes w_\alpha)$.

\begin{proposition}[Gaussian truncation]
  \label{prop:gaussian-truncation-error}
  Let $\alpha = \sqrt{\smash[b]{2/3}}$.
  For every $n \in \N$ it holds that
  \begin{equation*}
    \Bigg( \int_{-\infty}^\infty \int_{-\infty}^\infty (r(t, u) - r_{n}(t, u) )^2 w_\alpha(t) w_\alpha(u) \dif t \dif u \Bigg)^{1/2} = \frac{1}{\sqrt{2}} \, \frac{1}{3^n}.
  \end{equation*}
\end{proposition}
\begin{proof}
  As in the proof of \Cref{prop:matern-truncation-error}, we get
  \begin{equation*}
    \int_{-\infty}^\infty \int_{-\infty}^\infty (r(t, u) - r_{n}(t, u) )^2 w_\alpha(t) w_\alpha(u) \dif t \dif u = \sum_{m=n}^\infty \norm[0]{ \psi_{m} }_{\mathscr{L}_2(\R, w_\alpha)}^4.
  \end{equation*}
  By \Cref{prop:gaussian-basis-L2},
  \begin{equation*}
    \sum_{m=n}^\infty \norm[0]{ \psi_{m} }_{\mathscr{L}_2(\R, w_\alpha)}^4 = \sum_{m=n}^\infty \frac{4}{9^{m+1}} = \frac{1}{2} \, \frac{1}{9^n}.
  \end{equation*}
  This completes the proof.
\end{proof}

\section{Conclusion}
In this article, we have demonstrated that \Cref{thm:main-theorem} is a simple and powerful tool for constructing orthonormal expansions of translation-invariant kernels.
In particular, using the Cholesky factor of the Fourier transform of the kernel together with the Laguerre functions led to an interesting decomposition of the RKHS of the Mat\'ern kernel for half-integer smoothness parameters in terms of a finite dimensional space and a Hilbert space of functions vanishing at the origin.
This might be deemed unsatisfying, and a possible avenue to obtaining basis functions for Mat\'erns in a common space would be to investigate constructions based on the symmetric square-root.
The expansion for the Cauchy kernel was derived from the Fourier duality with the Mat\'ern kernel of smoothness $\alpha = 1/2$.
It remains an open problem to find a weighted $\mathscr{L}_2$ space in which the Cauchy basis functions are orthogonal.
For the Gaussian kernel, our construction is a means to reproduce certain Mercer expansions that are typically derived from Mehler's formula.

\section*{Acknowledgements}
FT was partially supported by the Wallenberg AI, Autonomous
Systems and Software Program (WASP) funded by the Knut and Alice Wallenberg Foundation, and
gratefully acknowledge financial support through funds from the Ministry of Science, Research and Arts of the State of Baden-Württemberg.TK was supported by the Academy of Finland postdoctoral researcher grant 338567 ``Scalable, adaptive and reliable probabilistic integration''.
Most of this article was written while TK was visiting the University of Tübingen in May 2022.

\bibliography{references}

\end{document}